\documentclass[11pt,a4paper]{article}
\usepackage[utf8]{inputenc}

\usepackage{hyperref}
\usepackage{amsmath,amsthm,amssymb,enumerate, bbm ,graphicx,color,caption,upgreek, float, tikz, subcaption,booktabs,longtable, appendix,graphics, pdfpages,rotating,mathtools, array, bm, blkarray,setspace,textcomp,pgfplots,todonotes, tcolorbox, mathtools}

\usepackage{arxiv}

\usepackage{algorithm}
\usepackage{algpseudocode}

\newtheorem{theorem}{Theorem}

\newtheorem{assumption}{Assumption}
\newtheorem{definition}{Definition}

\newtheorem{lemma}{Lemma}
\newtheorem{corollary}{Corollary}

\title{A predictor-corrector algorithm for semidefinite programming that uses the factor width cone}
\author{ \href{https://orcid.org/0000-0003-1771-0394}{\includegraphics[scale=0.06]{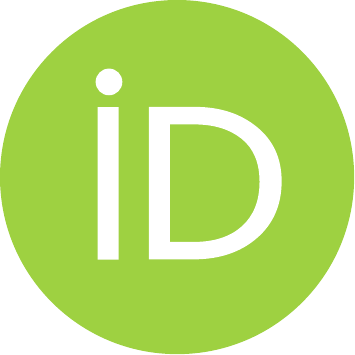}\hspace{1mm}Felix Kirschner} \\
		Department of Econometrics and OR\\
		Tilburg University\\
		Tilburg, Netherlands \\
	\texttt{f.c.kirschner@tilburguniversity.edu} \\
	\And
	\href{https://orcid.org/0000-0003-3377-0063}{\includegraphics[scale=0.06]{orcid.pdf}\hspace{1mm}Etienne de~Klerk} \\
	Department of Econometrics and OR\\
	Tilburg University\\
	Tilburg, Netherlands \\
	\texttt{e.deklerk@tilburguniversity.edu} \\
}
\date{}

\newcommand{\cnk}{{C_{k-1}^{n-1}}}
\newcommand{\dnk}{{C_{k-2}^{n-2}}}

\renewcommand{\headeright}{}
\renewcommand{\undertitle}{}
\renewcommand{\shorttitle}{An IPM for SDP using the factor width cone}

\hypersetup{
pdftitle={A predictor-corrector algorithm for semidefinite programming that uses the factor width cone},
pdfsubject={},
pdfauthor={Felix Kirschner, Etienne de Klerk},
pdfkeywords={},
}

\begin{document}
\maketitle

\begin{abstract}
	We propose an interior point method (IPM) for solving semidefinite programming problems (SDPs). The standard interior point algorithms used to solve SDPs work in the space of positive semidefinite matrices. Contrary to that the proposed algorithm works in the cone of matrices of constant \emph{factor width}. This adaptation makes the proposed method more suitable for parallelization than the standard IPM. We prove global convergence and provide a complexity analysis. Our work is inspired by a series of papers by Ahmadi, Dash, Majumdar and Hall, and builds upon a recent preprint by Roig-Solvas and Sznaier [arXiv:2202.12374, 2022].
\end{abstract}

\section{Introduction}\label{sec:intro}
Semidefinite programming problems (SDPs) are a generalization of linear programming problems (LPs). While capturing a much larger set of problems, SDPs are still being solvable up to fixed precision in polynomial time in terms of the input data \cite{Nesterov1994}; see \cite{Klerk2016} for the complexity in the Turing model of computation. In practice this is, however, more complicated. While we are able to solve linear programs with millions of variables and constraints routinely, SDPs become intractable already for a few tens of thousands of constraints and for $n\times n$ matrix variables of the order $n \approx 1,000$. The reason is that each iteration of a typical interior point algorithm for SDP requires $\mathcal{O}(n^3m+n^2m^2 + m^3)$ operations, where $n$ is the size of the matrix variable and $m$ is the number of equality constraints; see e.g. \cite{Alizadeh1998}. However, solving large instances of SDPs is  of growing interest, due to applications in power flow problems on large power grids, SDP-based hierarchies for polynomial and combinatorial problems, etc (see \cite{Lasserre2001,Zhang2021,Zohrizadeh2020}). In the following we will revisit a relaxation of a given SDP, where the cone of positive semidefinite matrices is replaced by a more tractable cone, namely the cone of matrices of constant factor width \cite{Boman2005}. The simplest examples of matrices of constant factor width are non-negative diagonal matrices (corresponding to linear programs), and scaled diagonally dominant matrices (corresponding to second order cone programming) \cite{Ahmadi2014}. We then review how iteratively rotating the cone and solving the given optimization problem over this new set leads to a non-increasing sequence of optimal values lower bounded by the optimum of the sought SDP. This iterative procedure, due to \cite{Ahmadi2015}, does not lead to a convergent algorithm. However, its essence can be used to construct a convergent predictor-corrector interior point method, as was done in \cite{Sznaier22}. Our paper is inspired by ideas from \cite{Ahmadi2014, Ahmadi2017, Ahmadi2015,Ahmadi2019,Sznaier22}. In particular, we will extend the results in \cite{Sznaier22}, and give a more concise complexity analysis in our extended setting.

\subsection{Iterative approximation scheme}\label{subsec:iterativescheme}
Let the set of symmetric $n \times n$ matrices be given by $\mathbb{S}^n$, where $n \in \mathbb{N}$ is a positive integer. We write $[m]$ for the set $\{1, 2, \ldots, m\}$, where $m \in \mathbb{N}$. Consider a set $\{ A_i \in \mathbb{S}^{n}: i \in [m] \}$ of symmetric data matrices and define the linear operator
\[
    \mathcal{A}(X) = (\langle A_1,X \rangle ,\ldots,\langle A_m,X\rangle) \in \mathbb{R}^m,
\]
where $\langle X, Y \rangle := \mathrm{tr}(XY)$ for $X,Y \in \mathbb{S}^n$.
Further, define for $b \in \mathbb{R}^m$ the affine subspace
\begin{equation}
    L = \{ X \in \mathbb{S}^n : \mathcal{A}(X) = b\}.
\end{equation}
Consider the following semidefinite program

\begin{equation}\label{IterativeSchemeprimal}
\begin{aligned}
    v^\ast_{\mathrm{SDP}} = \inf \left\{  \langle A_0, X \rangle : \mathcal{A}(X) = b, X \in \mathbb{S}^n_+\right\},
\end{aligned}
\end{equation}
which we assume to be strictly feasible. Replacing the cone of positive semidefinite (psd) matrices in \eqref{IterativeSchemeprimal} by a cone $\mathcal{K} \subseteq \mathbb{S}^n_+$, which is more tractable, leads to the following program

\begin{equation}\label{primalRel1}
    v_{\mathcal{K}} = \inf \left\{ \langle A_0, X \rangle : \mathcal{A}(X) = b , X \in \mathcal{K}, \text{ where } \mathcal{K} \subseteq \mathbb{S}^n_+ \right\}.
\end{equation}
Clearly, $v_{\mathcal{K}} \ge v^\ast_{\mathrm{SDP}}$. The quality of the approximation depends on the chosen cone $\mathcal{K}$. In \cite{Ahmadi2014}, while focusing on sums-of-squares optimization the authors consider the cones of \emph{diagonally dominant} and \emph{scaled diagonally dominant} matrices. Ahmadi and Hall developed the idea of replacing the psd cone by a simpler cone further in \cite{Ahmadi2015}, leveraging an optimal solution of the relaxation. Essentially, the idea is as follows. Define the feasible set for \eqref{IterativeSchemeprimal} as
\[
    \mathcal{F}_{\mathrm{SDP}} = \left \{ X \succeq 0 : \mathcal{A}(X) = b \right \}.
\]
We will consider a sequence of strictly feasible points for \eqref{primalRel1}, denoted by $X_\ell$ for $\ell = 0, 1 \ldots$. Since $X_\ell \succeq 0$, the matrix $X_\ell^{1/2}$ is well-defined. One can \emph{update} the data matrices in the following way
\[
    A_i^{(\ell)} = X^{1/2}_\ell A_i X^{1/2}_\ell \quad (i \in \{0,1,\ldots,m\},\, \ell = 0,1,  \ldots),
\]
giving rise to a new linear operator
\[
    \mathcal{A}^{(\ell)}(X) = (\langle A_1^{(\ell)}, X\rangle ,\ldots,\langle A_m^{(\ell)}, X\rangle ) \in \mathbb{R}^m.
\]
We may also refer to this operation as \emph{rescaling} with respect to $X_\ell$.
Via this rescaling one obtains the following sequence of reformulations of \eqref{IterativeSchemeprimal}
\begin{equation}\label{SDPell}
        v^{\ast}_\mathrm{SDP} =\mathrm{min} \left\{ \langle A_0^{(\ell)}, X \rangle :
          \mathcal{A}^{(\ell)}(X) = b,
          X \in \mathbb{S}^n_+\right\},
\end{equation}
whose feasible set we define as
\[
    \mathcal{F}_{\mathrm{SDP}_\ell} =  \left \{ X \succeq 0 : \mathcal{A}^{(\ell)}(X) = b \right \}.
\]
For each $\ell$ the identity matrix is feasible, i.e., we have $X = I \in \mathcal{F}_{\mathrm{SDP}_\ell}$. To see this, note that for all $i \in [m]$ we have
\[
 \langle A_i^{(\ell)}, I \rangle = \langle \left(X_\ell \right)^{\frac{1}{2}}A_i \left(X_\ell \right)^{\frac{1}{2}}, I \rangle = \langle A_i ,X_\ell \rangle = b_i.
\]
Similarly, the identity leads to the same objective value in \eqref{SDPell} as $X_\ell$ in \eqref{primalRel1}. Let $X_0$ be an optimal solution to \eqref{primalRel1}. Rescaling with respect to $X_0$ we find by the same reasoning that $v_{\mathcal{K}}^{(0)} \le v_{\mathcal{K}}$, where
\begin{equation}\label{relaxell}
        v^{(\ell)}_\mathcal{K} =\mathrm{min}\left\{ \langle A_0^{(\ell)}, X \rangle :
          \mathcal{A}^{(\ell)}(X) = b ,
         X \in \mathcal{K} \right\}.
\end{equation}

Reiterating this procedure leads to a non-increasing sequence of values $\left\{ v^{(\ell)}_{\mathcal{K}} \right\}_{\ell \in  \mathbb{N}}$ lower bounded by $v^{\ast}_{\mathrm{SDP}}$. Unfortunately, this procedure does not converge to the true optimum of \eqref{IterativeSchemeprimal} in general, as mentioned in \cite{Sznaier22}. Indeed, it can happen that $\liminf_{\ell \rightarrow \infty} v_{\mathcal{K}}^{(\ell)} > v^*_{\mathrm{SDP}}$. The rest of this paper is devoted to the development and analysis of an algorithm, which converges to the optimal value $v^\ast_{\mathrm{SDP}}$. We thereby generalize results from \cite{Sznaier22}.

\subsection*{Outline of the paper}
This paper is conceptually divided into two parts. The first part contains sections \ref{sec:intro} and \ref{sec:ipm} and is devoted to introducing the setting as well as the algorithm. Our aim with the first part is to convey the concept in a comprehensible way. The second part consists of the remaining sections \ref{sec:barriers}-\ref{sec:discussion}. It is more technical and contains the derivation of objects used in the algorithm as well as the formal complexity analysis.

\subsection{The factor width cone}

Fix $n \in \mathbb{N}$. The cone of $n\times n$ matrices of \emph{factor width} $k$, denoted by FW$_n(k)$, is defined as
\[
    \mathrm{FW}_n(k) = \left\{ Y \in \mathbb{S}^n : Y = \sum_{i \in \mathbb{N}} x_i x_i^T \text{ for } x_i \in \mathbb{R}^n,\, \text{supp}(x_i) \le k\,, \forall i\,\right\}.
\]
The notion of factor width was first used in \cite{Boman2005} where the authors proved that $\mathrm{FW}_n(2)$ is the cone of scaled diagonally dominant matrices. Trivially, $\mathrm{FW}_n(1)$ is the cone of non-negative $n \times n$ diagonal matrices. Clearly, we have that
\[
    \mathrm{FW}_n(k) \subseteq \mathrm{FW}_n(k+1) \subseteq \mathbb{S}^n_+ \quad \forall k \in [n-1].
\]
Moreover, $\mathrm{FW}_n(n) = \mathbb{S}^n_+$. It is easy to see these cones are proper. As they define an inner approximation of the cone $\mathbb{S}^n_+$ we may use them in the aforementioned iterative scheme. Define
\[
\mathbb{S}^{(n,k)} := \underbrace{\mathbb{S}^k \times \dots \times \mathbb{S}^k}_{{{n}\choose{k}}\text{-times}} \text{ and } \mathbb{S}_+^{(n,k)} := \underbrace{\mathbb{S}^k_+ \times \dots \times \mathbb{S}^k_+}_{{{n}\choose{k}}\text{-times}}.
\]
An optimization problem over the cone $\mathrm{FW}_n(k)$ may be formulated as an optimization problem over the cone product $\mathbb{S}_+^{(n,k)}$. To see this we need to consider principal submatrices. For a matrix $S \in \mathbb{R}^{n \times n}$ we define the principal submatrix $S_{J,J}$ for $J \subseteq [n]$ to be the restriction of $S$ to rows and columns whose indices appear in $J$. Further, for a set $J = \{ i_1, \dots, i_{\vert J\vert} \} \subseteq [n]$ and a matrix $S \in \mathbb{R}^{\vert J\vert \times \vert J\vert}$ we define the $n \times n$ matrix $S_J^{\rightarrow n}$ as follows for $i,j \in [n]$
\begin{equation}\label{Yarrow}
    \left( S_J^{\rightarrow n} \right)_{i,j} = \begin{cases}
        S_{k,l} &\text{if }i = i_k, j = i_l\\
        0 &\text{otherwise.}
    \end{cases}
\end{equation}
In other words, $S_J^{\rightarrow n}$ has $S_J$ as principal sub-matrix indexed by $J$, and zeros elsewhere. Now, to write a program over $\mathrm{FW}_n(k)$ as an SDP note the following lemma. 
\begin{lemma}\label{lemmaXpsdRep}
    For any $X \in \mathrm{FW}_n(k)$ we have that
    \[
        X = \sum_{\vert J\vert = k} Y_J^{\rightarrow n}
    \]
    for suitable $Y_J \in \mathbb{S}^k_+$ and $J \subseteq [n], \vert J\vert = k$.
\end{lemma}
\begin{proof}
    The proof is straightforward and omitted for the sake of brevity.
\end{proof}
Thus, we can write

\begin{equation}\label{fwProg}
    \inf \left\{ \langle C, X \rangle : \mathcal{A}(X) = b, X \in \mathrm{FW}_n(k) \right\}
\end{equation}

as

\begin{equation}\label{progFW}
    \inf \left\{ \sum_{\vert J\vert=k}\langle C_{J,J}, Y_J \rangle : \sum_{\vert J\vert=k} \langle (A_i)_{J,J}, Y_J \rangle = b_i, Y_J \in \mathbb{S}^k_+, \; \forall \vert J\vert = k\right\}.
\end{equation}

It is straightforward to show that the dual cone is given by

\[
    \mathrm{FW}_n(k)^\ast = \{ S \in \mathbb{S}^n : S_{J,J} \succeq 0 \text{ for } J \subseteq [n], \vert J \vert = k \}.
\]
The dual cone has been studied in the context of semidefinite optimization in \cite{Blekherman2022}, where it was shown that the distance of $\mathrm{FW}_n(k)^\ast$ and $\mathbb{S}^n_+$ in the Frobenius norm can be upper bounded by $\frac{n-k}{n+k-2}$ for matrices of trace $1$. For $k\ge 3n/4$ and $n\ge 97$ this bound can be improved to $O(n^{-3/2})$ (see \cite{Blekherman2022}).

\section{Interior point methods and the central path}\label{sec:ipm}

Interior point methods (IPMs) are among the most commonly used algorithms to solve conic optimization problems in practice. Notable software for IPMs include Mosek \cite{mosek}, CSDP \cite{Borchers1999}, SDPA \cite{SDPA,Fujisawa2002}, SeDuMi \cite{sedumi} and SDPT3 \cite{Toh1999}. In the remainder of this section, we will closely follow the notation used in \cite{Renegar2001}, since we will make use of several results from this book. Consider the following conic optimization problem for a proper convex cone $\mathcal{K} \subset \mathbb{R}^n$:
\begin{equation*}
     \min \left\{ \langle c, x \rangle : \langle a_i, x \rangle = b_i, i \in [m], x \in \mathcal{K}\right\}.
\end{equation*}
In IPMs the cone membership constraint is replaced by adding a convex penalty function $f$ to the objective. This function $f$ is a so-called self-concordant barrier function. Loosely speaking, the function $f$ returns larger values the closer the input is to the boundary of the cone and tends to infinity as the boundary is approached. In order to formally define self-concordant barrier functionals, let $f : \mathbb{R}^n \supset D_f \rightarrow \mathbb{R}$ be such that its Hessian $H(x)$ is positive definite (pd) for all $x \in D_f$. With respect to this function, we can define a \emph{local inner product} as follows
\[
    \langle u,v \rangle_x := \langle u, H(x)v \rangle,
\]
where $u,v \in \mathbb{R}^n$ and $\langle \cdot, \cdot \rangle$ is some reference inner product. Let $B_x(y,r)$ be the open ball centered at $y$ with radius $r>0$ whose radius is measured by $\vert\vert\cdot\vert\vert_x$, i.e., the norm arising from the local inner product at $x$.
\begin{definition}{(see \cite[§~2.2.1]{Renegar2001})}
    A functional $f$ is called \emph{(strongly non-degenerate) self-concordant} if for all $x \in D_f$ we have that $B_x(x,1) \subset D_f$ and whenever $y \in B_x(x,1)$ we have
    \[
        1-\vert\vert y-x\vert\vert_x  \le \frac{\vert\vert v \vert\vert_x}{\vert\vert v \vert\vert_x} \le \frac{1}{1-\vert\vert y-x \vert\vert_x} \text{ for all }v \neq 0.
    \]
    A functional $f$ is called a \emph{self-concordant barrier functional} if $f$ is self-concordant and additionally satisfies
    \[
        \vartheta_f := \sup_{x \in D_f} \vert\vert H(x)^{-1}g(x)\vert \vert_x^2<\infty,
    \]
    where $g(x)$ is the gradient of $f$.
\end{definition}
We refer to $\vartheta_f$ as the complexity value of $f$ (see \cite[p. 35]{Renegar2001}), which will become crucial in our complexity analysis.
Henceforth, let $f$ be a self-concordant barrier functional for $\mathcal{K}$ and consider the following family of problems for positive $\eta \in \mathbb{R}_+$

\begin{equation}\label{ipm_dummy}
\begin{aligned}
    z_{\eta} = \mathrm{argmin}\;  & \eta \, \langle c, x \rangle + f(x) \\ 
    \text{s.t. } &\langle a_i, x \rangle = b_i \quad i \in [m]. \\
\end{aligned}
\end{equation}
The minimizers $z_\eta$ of \eqref{ipm_dummy} define a curve, parametrized by $\eta$ in the interior of $\mathcal{K}$. This curve is called the \emph{central path}. For $\eta \rightarrow \infty$ one can show that $z_\eta \rightarrow x^\ast$. Interior point methods work by subsequently approximating a sequence of points $\{ z_{\eta_i} : i = 1, \dots, N\}$ on the central path, where $\eta_1 < \eta_2 < \ldots$ such that $z_{\eta_N}$ is within the desired distance to the optimal solution. The type of interior point method we consider is an adaptation of the \emph{predictor-corrector method} (see \cite[§~2.4.4]{Renegar2001}). This method uses the ordinary \emph{affine scaling direction} to produce a new point inside the cone with decreased objective value. Afterwards, a series of corrector steps is performed to obtain feasible solutions with the same objective value that lie increasingly close to the central path. Interior point methods typically rely on Newton's method in each step, where the convergence rate depends on the so-called Newton decrement. 
\begin{definition}
If $f:\mathbb{R}^n \rightarrow \mathbb{R}$ has a gradient $g(x)$ and positive definite Hessian $H(x) \succ 0$ at a point $x$ in its domain, then
the Newton decrement of $f$ at $x$ is defined as
\[
\Delta(f,x) = \sqrt{\langle g(x), H^{-1}(x)g(x)\rangle}.
\]
\end{definition}
For self-concordant functions $f$, a sufficiently small value of $\Delta(f,x)$, e.g., $\Delta(f,x) < 1/9$,  implies that $x$ is close to the minimizer of $f$ (cf. \cite[Theorem 2.2.5]{Renegar2001}).

Suppose we are given a starting point $x_0$, which is close to $z_{\eta_0}$ for some $\eta_0 \in \mathbb{R}$. The affine-scaling direction is given by $-c_{x_0} := - H(x_0)c$ and points approximately tangential to the central path in the direction of decreasing the objective value $\langle c, x \rangle$ ($-H(z_{\eta_0})c$ is exactly tangential to the central path). The predictor step moves from $x_0$ a fixed fraction $\sigma \in (0,1)$ of the distance towards the boundary of the feasible set in the affine-scaling direction, thereby producing a new point $x_1$ satisfying $\langle c, x_1 \rangle < \langle c, x_0 \rangle$. The new point $x_1$ is not necessarily close to the central path. The algorithm then proceeds to produce a sequence of feasible points $x_2, x_3, \dots$ satisfying $\langle c, x_1\rangle = \langle c , x_i \rangle$ for $i = 2, 3, \dots$ while each $x_i$ for $i = 2, 3, \dots$ is closer to the central path than its predecessor $x_{i-1}$. In other words, the algorithm targets the point $z_{\eta_1}$ on the central path with the same objective value as $x_1$ and produces a sequence of points converging to $z_{\eta_1}$. Once an $x_j$ is found such that $\Delta(f, x_j) < 1/9$, the next predictor step is taken. This procedure is repeated until an $\varepsilon$-optimal solution is found. The corrector phase works by minimizing the self-concordant barrier restricted to the feasible affine space intersected with the set of all $x \in \mathbb{R}^n$ such that $\langle c, x\rangle = \langle c, x_i\rangle$, where $x_i$ is the point produced by the most recent predictor step. This minimization problem is solved iteratively by performing line searches along the direction given by the Newton step for the restricted functional. We provide a visualization of the predictor-corrector method in Figure \ref{fig:pred_corr}.

\subsection*{Newton decrements for functions restricted to subspaces} 

If a self-concordant function $f$ is restricted to a (translated) linear subspace $L$, and denoted by $f_{\vert L}$, then the Newton decrement at $x$ becomes
\[
\Delta\left(f_{\vert L},x\right) = {\vert \vert P_{L,x}H^{-1}(x)g(x)\vert \vert _x},
\]
where $\|\cdot\|_x$ is the norm induced by the inner product $\langle u,v\rangle_x = \langle u, H(x)v \rangle$, and $P_{L,x}$ is the orthogonal projection onto
$L$ for the $\|\cdot\|_x$ norm; see \cite[§~1.6]{Renegar2001}.

Note that we have
    \begin{align*}
        \Delta(f, x) & = \langle g(x), H^{-1}(x)g(x)\rangle^{1/2} = \langle g(x), - n(x)\rangle^{1/2} \\
        & = \langle n(x),   n(x)\rangle_{x}^{1/2} = \vert \vert n(x)\vert \vert_{x}  = \sup_{\vert \vert d \vert \vert_{x} = 1} \langle d, n(x)\rangle_{x},
    \end{align*}
    where $n(x)$ is the Newton step at $x$, i.e., $n(x)= -H(x)^{-1}g(x)$.
    Hence, restricting the function $f$ to a subspace $L$ we find
    \begin{equation}\label{varNorm}
        \begin{aligned}
        \Delta\left(f_{\vert _L}, x\right) & = \sup_{\vert \vert d \vert \vert _{x = 1}} \langle d, P_{L,x} n(x) \rangle_{x} = \sup_{\stackrel{\vert \vert d \vert \vert _{x = 1}}{d \in L}} \langle d, n(x) \rangle_{x} \\
        &= \sup_{0 \neq d \in L} \frac{\langle d, n(x) \rangle_{x}}{\vert \vert d\vert \vert _{x}} \ge \frac{\langle d, n(x) \rangle_{x}}{\vert \vert d\vert \vert _{x}} \text{ for all } d \in L\setminus \{0 \}.
        \end{aligned}
    \end{equation}

\begin{figure}[]
    \includegraphics[width=0.9\textwidth]{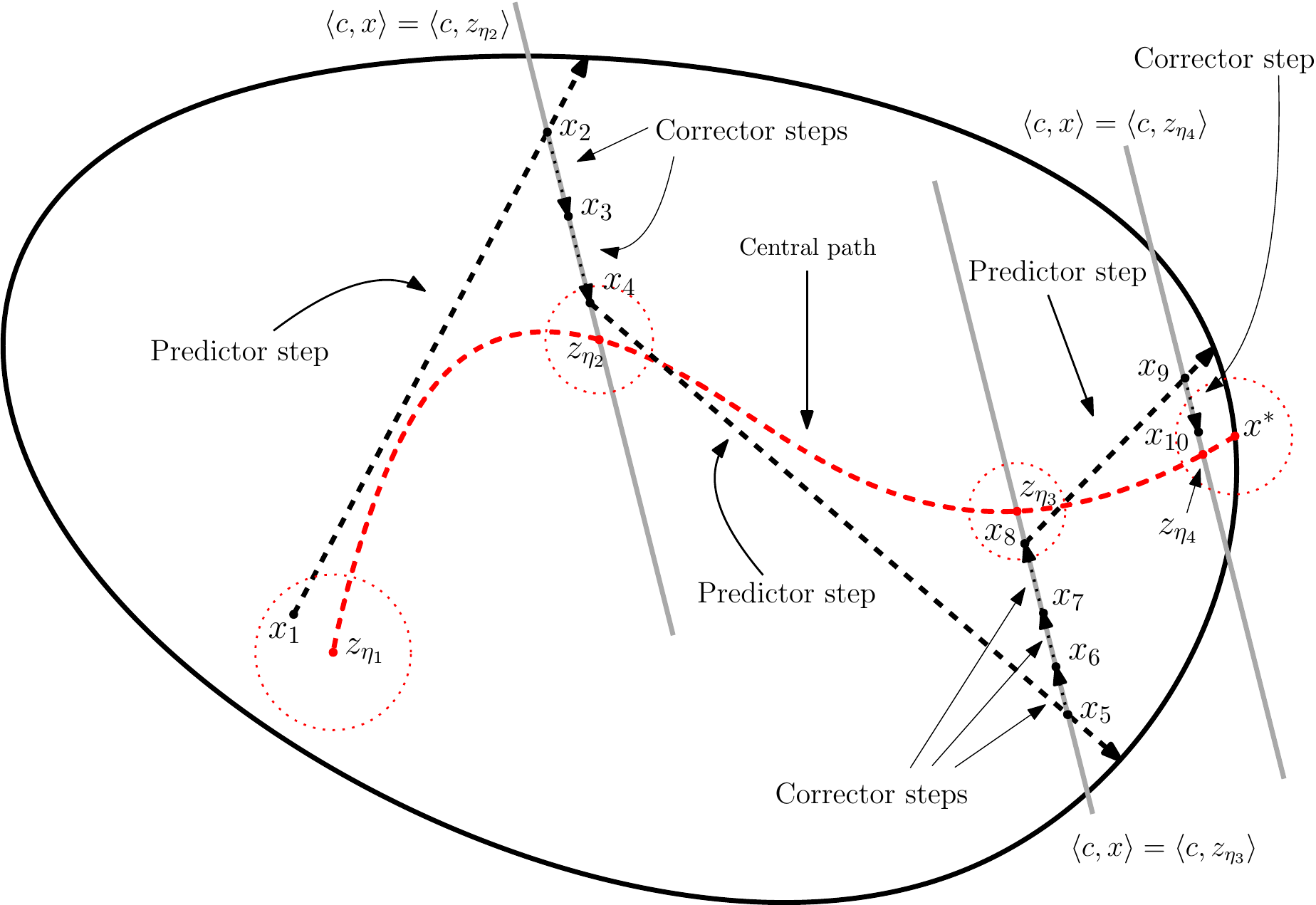}
    \caption{Visualization of predictor-corrector method. Initial feasible solution close to central path (red) is given by $x_1$. Algorithm performs predictor step returning $x_2$. Corrector steps are taken until point close enough to central path ($x_4$) is found. Next predictor step returns $x_5$. Corrector steps are taken until $x_8$ is found, which is close enough to central path to perform next predictor step returning $x_9$. After one corrector step the final point $x_{10}$ is $\varepsilon$-close to $x^\ast$.}\label{fig:pred_corr}
\end{figure}

\subsection{A predictor-corrector method}\label{sec:algorithm}
In this subsection we propose our algorithm which makes use of the rescaling introduced in section \ref{subsec:iterativescheme}. Our aim is to provide a comprehensible exposition, while the details are postponed to the second part of the paper, beginning with section \ref{sec:barriers}.

Algorithm \ref{alg:pc_ipm} is an adaption of the predictor-corrector method as described in \cite[§~2.2.4]{Renegar2001}. Before describing the algorithm in detail we fix some notation. Let
\[
    \mathcal{Y} = \left\{ Y_J \in \mathbb{S}^k : J \subset [n], \vert J \vert = k \right\}
\]
be a collection of ${{n}\choose{k}}$ matrices of size $k \times k$.
We define the operator $\Psi$ as
\[
    \Psi(\mathcal{Y}) = \sum_{ \vert J \vert = k}Y^{\rightarrow n}_J,
\]
where we made use of the notation defined in \eqref{Yarrow}.
Hence, if $\mathcal{Y}$ is a collection of positive semidefinite $k\times k$ matrices, then $\Psi(\mathcal{Y}) \in \mathrm{FW}_n(k)$. Furthermore, let
\begin{equation}\label{Y0}
    \mathcal{Y}_0 = \{ Y_J = 1/\cnk I_{k \times k} : J \subset [n], \vert J \vert = k \},
\end{equation}
where we denote for $n,k \in \mathbb{N}$ the binomial coefficient as ${{n}\choose{k}} =: C_k^n$, so that $\Psi(\mathcal{Y}_0) = I$. Now let $X_\ell$ be a strictly feasible solution to a problem of form \eqref{IterativeSchemeprimal} and rescale the data matrices with respect to $X_\ell$. Recall the feasible set of the resulting SDP is given by
\begin{equation}\label{Lell}
    L_\ell = \{ X \in \mathbb{S}^n : \mathcal{A}^{(\ell)}_0(X) = b \}.
\end{equation}
Likewise, the feasible set of the factor width relaxation written over $\mathbb{S}^{(n,k)}_+$ (cf. \eqref{progFW}) can be written as
\begin{equation}\label{Lpsiell}
    L^{\Psi}_{\ell}  = \{ \mathcal{Y} \in \mathbb{S}^{(n,k)} : (\mathcal{A}^{(\ell)} \circ \Psi) (\mathcal{Y}) = b\}.
\end{equation}
Note that $I \in L_\ell$ and $\mathcal{Y}_0 \in L^{\Psi}_\ell$. We emphasize that, by definition, for any element $\mathcal{Y} \in L_{\ell}^{\Psi}$ we have $\Psi(\mathcal{Y}) \in L_\ell$.

\subsection*{Main method}
The algorithm requires a feasible starting point $X_0$ close to the central path, which is used in the first rescaling step. We also require an $\varepsilon > 0$, i.e., our desired accuracy as well as a $\sigma \in (0,1)$ used in the predictor step. In the following let $f^{\mathrm{FW}(k)}$ be a self-concordant barrier function for $\mathbb{S}^{(n,k)}_+$ (we postpone its derivation to section \ref{sec:barriers}, for now we assume it exists and is efficiently computable). In the algorithm we denote the restriction of $f^{\mathrm{FW}(k)}$ to the subspace $\mathrm{null}(L^{\Psi}_\ell)$ by $f^{\mathrm{FW}(k)}_{\vert\mathrm{null}(L^{\Psi}_\ell)}$. The algorithm initializes $\ell = 0$. The outer while loop repeats until an $\varepsilon$ optimal solution is found. If after rescaling with respect to $X_\ell$ the Newton decrement at $\mathcal{Y}_0$ satisfies
\[
    \Delta\left(f^{\mathrm{FW}(k)}_{\vert \mathrm{null}(L^{\Psi}_\ell)}, \mathcal{Y}_0\right) \le 1/14
\]
the predictor subroutine is called. Here, the affine-scaling direction is projected onto the null space of $L_\ell^{\Psi}$, call it $\mathcal{Z}$. Clearly, $\mathcal{Y}_0 + s \mathcal{Z} \in L_\ell^{\Psi}$ for all $s \in \mathbb{R}$. Then the subroutine computes
\[
    s^\ast = \sup \left\{ s : \mathcal{Y}_0 - s \mathcal{Z} \in \mathbb{S}^{(n,k)}_+ \right\}
\]
which provides the necessary notion of distance to the boundary in terms of $\mathcal{Y}_0$ and $\mathcal{Z}$. The returned point $\mathcal{Y}_\ell := \mathcal{Y}_0 + \sigma s^\ast \mathcal{Z}$ is feasible and decreases the objective value, as shown in section \ref{sec:complexity}. If the Newton decrement is not small enough, the corrector subroutine is called. Let $v_\ell = \langle A_0, X_\ell \rangle$, i.e., the objective value of the previous iteration, and define

\[
    L_\ell^{\Psi}(v_\ell) = \{ \mathcal{Y} \in \mathbb{S}^{(n,k)}_+ : \langle A_0 , \Psi(\mathcal{Y}) \rangle = v_\ell, \mathcal{A}^{(\ell)}(\Psi(\mathcal{Y}))=b\}.
\]
Let $x_0 := \mathcal{Y}_0$. Denote by $n_{\vert L^{\Psi}_\ell(v_{\ell})}(x_i)$ the Newton step of $f^{\mathrm{FW}(k)}_{\vert L^{\Psi}_\ell(v_{\ell})}$ at a point $x_i$. The corrector step now computes
\[
    x_{i+1} = \mathrm{argmin}_{t} f^{\mathrm{FW}(k)}\left( x_i + t n_{\vert L^{\Psi}_\ell(v_{\ell})}(x_i)\right)
\]
until $x_{i+1}$ is close enough to the central path of the rescaled problem over $\mathbb{S}^{(n,k)}_+$ and returns $\mathcal{Y}_\ell := x_{i+1}$. We will prove in section \ref{sec:relating_barriers} how this leads to a decrease in distance to the central path of original SDP. Note that multiple calls of the corrector step may be necessary as after rescaling the Newton decrement might not be small enough anymore. However, as we prove later on, the maximum number of corrector step can be bounded in terms of the problem data.
Let $\mathcal{Y}_\ell$ be the point returned by one of the subroutines. We set
\[
    X_{\ell+1} = X_{\ell}^{1/2} \Psi(\mathcal{Y}_\ell)X_{\ell}^{1/2}.
\]
Then
\[
    \langle A^{(\ell+1)}_i, I \rangle = \langle A^{(\ell)}_i, \Psi(\mathcal{Y}_\ell) \rangle = \langle A_i , X_{\ell+1} \rangle
\]
for all $i = 0,1, \dots, m$. 

\subsection*{Termination criterion}
In the predictor as well as in the corrector subroutine we solve a linear system for $y \in \mathbb{R}^m$. The solution of this linear system may be interpreted as a dual feasible solution provided the current iterate is sufficiently close to the central path. Hence, we can approximate the duality gap of our problem by calculating the difference
\[
    \langle A_0, X_\ell \rangle - y^Tb \ge 0,
\]
where $y$ is calculated in every subroutine call. We may use this as a termination criterion. Once the duality gap falls below some $\varepsilon > 0$ chosen beforehand, we terminate with an $\varepsilon$ optimal solution.



\begin{algorithm}
    \caption{Predictor-Corrector SDP algorithm using FW$_n(k)$}\label{alg:pc_ipm}
    \begin{algorithmic}
    \Require $\varepsilon > 0,\, \sigma \in (0,1), \, X_0$ close to CP
    \State $\ell \gets 0$
    \While{Duality gap $> \varepsilon$}
            \State $A_i^{(\ell)} \gets \left(X_{\ell}\right)^{1/2}A_i \left(X_{\ell}\right)^{1/2}$, for $i = 0,1,..., m$
            \If{$\Delta\left(f_{\vert \mathrm{null}(L^{\Psi}_\ell)}^{\mathrm{FW}(k)}, \mathcal{Y}_0\right) \le \frac{1}{14}$}
                \State $\mathcal{Y}_\ell \gets {\tt Predictor\_Step}(\mathcal{A}^{(\ell)}, A_0^{(\ell)}, \sigma)$
            \Else
                \State $\mathcal{Y}_{\ell} \gets {\tt Corrector\_Step}(\mathcal{A}^{(\ell)}_0 \circ \Psi, f^{\mathrm{FW}(k)}, \mathcal{Y}_0)$
            \EndIf
            \State $X_{\ell+1} \gets \left(X_{\ell}\right)^{1/2}\Psi(\mathcal{Y}_{\ell})\left(X_{\ell}\right)^{1/2}$
            \State $\ell \gets \ell + 1$
    \EndWhile \\
    \Return $X_{\ell}$
    \end{algorithmic}
    \end{algorithm}

\begin{algorithm}
    \caption{Subroutine {\tt Predictor\_Step}}\label{alg:subroutine_predictor_method}
    \begin{algorithmic}
        \Require $\mathcal{A}, A_0, \sigma \in (0,1)$
            \State Solve for $y$: $\mathcal{A}A_0 = \mathcal{A}\mathcal{A}^\ast y$
            \State $ \mathcal{Z} =\Psi^{\dagger}(\mathcal{A}^\ast y-A_0)$
            \State $s^\ast \gets \sup \{ s : \mathcal{Y}_0 - s \mathcal{Z} \in \mathrm{FW}_n(k)\}$
            \State $\mathcal{Y} \gets \mathcal{Y}_0 - \sigma s^\ast \mathcal{Z}$ \\
            \Return $\mathcal{Y}$
        \end{algorithmic}
\end{algorithm}

\begin{algorithm}
    \caption{Subroutine {\tt Corrector\_Step}}\label{alg:subroutine_corrector_method}
    \begin{algorithmic}
        \Require $\mathcal{A}, f,  x^{(0)} : \Delta\left(f\vert_L,x^{(0)}\right) > \frac{1}{14},\, (L = \mathrm{null}(\mathcal{A}))$
            \State $j \gets 0$
            \While{$ \left(f_{\vert_L}
            ,x^{(j)}\right)> \frac{1}{14}$}
            \State Solve for $y$: $\mathcal{A}H(x^{(j)})^{-1}\mathcal{A}^\ast y = \mathcal{A}H(x^{(j)})^{-1}g(x^{(j)})$
                \State $n_{\vert_L}(x^{(j)})\gets H(x^{(j)})^{-1}\left( \mathcal{A}^\ast y - g(x^{(j)})\right)$
                \State $x^{(j+1)} \gets \mathrm{argmin}_t \; f\left( x^{(j)} + t n_{\vert_L}(x^{(j)}) \right)$
                \State $j \gets j+1$
            \EndWhile \\
        \Return $x^{(j-1)}$
    \end{algorithmic}
\end{algorithm}

\section{Barrier functionals for $\mathbb{S}^{n}_+$ and $\mathrm{FW}_n(k)$}\label{sec:barriers}

In this section we derive the self-concordant barrier functional for the cone $\mathbb{S}^{(n,k)}_+$ which is used in the algorithm. Note that the ordinary self-concordant barrier for $\mathbb{S}^n_+$ is given by $f^{\mathrm{SDP}}(X) = - \log (\det (X))$. We will emphasize parallels to the work of Roig-Solvas and Sznaier \cite{Sznaier22}.

In order to construct a self-concordant barrier function for our underlying set, we introduce the notions of hyper-graphs and edge colorings as well as a well-known result about these objects.

\begin{definition}
    A hyper-graph $\mathcal{H} = (V,E)$ consists of a set $V = \{1, \dots, n\}$ of vertices and a set of hyper-edges $E \subseteq \{ J \subseteq V : \vert J \vert \ge 2\}$, which are subsets of the vertex set $V$. If all elements in $E$ contain exactly $k$ vertices, we call the corresponding hyper-graph \emph{$k$-uniform}.
\end{definition}

\begin{definition}
    Let $\mathcal{H} = (V,E)$ be a hyper-graph. A proper hyper-edge coloring with $m$ colors is a partition of the hyper-edge set $E$ into $m$ disjoint sets, say $E= \cup_{i \in [m]} S_i$ such that $S_i \cap S_j = \emptyset$ if $i \neq j$, i.e., two hyper-edges that share a vertex are not in the same set. In other words, a proper hyper-edge coloring assigns a color to every hyper-edge such that, if a given vertex appears in two different hyper-edges, they have different colors.
\end{definition}

\begin{theorem}[Baranyai's theorem \cite{Baranyai1975}]\label{Baranyai}
    Let $k,n \in \mathbb{N}$ such that $k \vert n$ and let $K^n_k$ the complete $k$-uniform hyper-graph on $n$ vertices. Then there exists a proper hyper-edge coloring using $\cnk$ colors.
\end{theorem}

In \eqref{progFW} we wrote a program over $\mathrm{FW}_n(k)$ as an equivalent program over the cone product $\mathbb{S}^{(n,k)}_+$. The algorithm uses a self-concordant barrier function over said cone product. The mapping $\Psi$ from $\mathbb{S}_+^{(n,k)}$ to $\mathrm{FW}_n(k)$ is surjective, but not bijective, since multiple elements in the former may give rise to the same element in the latter set. 
\begin{assumption}
    Throughout we will assume $k \vert n$ for some $n \in \mathbb{N}$ and $2 \le k \in \mathbb{N}$.
\end{assumption}

In the following we will let $\mathcal{J} = \{J \subset [n] : \vert J \vert = k\}$ and
\[
    \mathcal{Y} = \left\{ Y_J : J \in \mathcal{J} \right\}
\]
be a collection of ${{n}\choose{k}}$ matrices of size $k \times k$.
We recall the operator $\Psi$ is defined as
\[
    \Psi(\mathcal{Y}) = \sum_{J \in \mathcal{J}}Y^{\rightarrow n}_J.
\]
The following generalizes Lemma 4.4 in \cite{Sznaier22}, where a similar result is proved for $k=2$. It will be crucial in our analysis as it allows us to compare the values taken by the barrier functionals on $\mathbb{S}^{(n,k)}_+$ and $\mathbb{S}^n_+$ at $\mathcal{Y}$ and $\Psi(\mathcal{Y})$, respectively.
\begin{lemma}\label{lemma1}
    Let
    \[
    f^{\mathrm{FW}(k)}(\mathcal{Y})  = -\sum_{J \in \mathcal{J}} \log(\det(Y_J))\, , \, \mathcal{Y} \in \mathrm{int}\left( \mathbb{S}^{(n,k)_+} \right).
    \]
    The barrier $f^{\mathrm{FW}(k)}(\mathcal{Y})$ is self-concordant on $\mathrm{int}\left( \mathbb{S}^{(n,k)_+} \right)$. Furthermore, if $X = \Psi(\mathcal{Y})$ then
    \begin{align*}
        f^{\mathrm{FW}(k)}(\mathcal{Y}) & \ge -\cnk \log (\det(X)) + n \cnk \log \left(\cnk\right) \\
                                        & =: \cnk f^{\mathrm{SDP}}(X) +n \cnk \log \left(\cnk\right).
    \end{align*}
\end{lemma}
Let us emphasize here that $f^{\mathrm{FW}(k)}$ is a self-concordant barrier for $\mathbb{S}^{(n,k)}_+$ not $\mathrm{FW}_n(k)$. Before proving Lemma \ref{lemma1} we need an auxiliary result which extends Lemma A.1 from \cite{Sznaier22} to general values of $k$ such that $k\vert n$. To prove it we will make use of Theorem \ref{Baranyai}.

\begin{lemma}
    Consider the set $\mathcal{Y} = \{ Y_J : J \in \mathcal{J} \}$ consisting of positive definite $k \times k$ matrices and let $X = \Psi(\mathcal{Y}) \in \mathrm{FW}_n(k)$. Then there exists a set of $\cnk$ matrices $Z_i \succ 0$ of size $n \times n$ such that $X = \sum_{i = 1}^{\cnk}Z_i$ and $f^{\mathrm{FW}(k)}(\mathcal{Y})=- \sum_{i=1}^{\cnk} \log(\det(Z_i))$.
\end{lemma}

\begin{proof}
    Let $K^n_k$ be the complete $k$-uniform hyper-graph on $n$ vertices. We can identify each hyper-edge $\{i_1,i_2,\dots, i_k\}\subset [n]$ in $K^n_k$ with exactly one element $Y_J \in \mathcal{Y}$, namely the one where $\{i_1,i_2,\dots, i_k\} =J$. Let $\{ S_1, \dots, S_{\cnk} \}$ be a hype-edge coloring of $K^n_k$. Define $\mathcal{Y}_i := \{Y_J : J \in S_i \}$ and set $Z_i := \Psi(\mathcal{Y}_i)$. Then $X = \sum_{i=1}^{\cnk}Z_i$ since $S_i \cap S_j = \emptyset$ for $i \neq j$ and $\cup_{i} S_i = \mathcal{J}$. Since $f^{\mathrm{FW}(k)}(\mathcal{Y})$ is finite, we know that $Z_i \succ 0$. Moreover, since each $S_i$ induces a perfect matching, there exists a permutation matrix $P_i$ for every $i = 1, \dots, \cnk$ such that $P_i Z_i P_i^T$ is a block-diagonal matrix with blocks $Y_J$ on the diagonal for $J \in S_i$. From this we find
    \[
        \log(\det(Z_i)) = \log(\det(P_i Z_i P_i^T)) = \sum_{J \in S_i}\log(\det(Y_J)).
    \]
    Hence,
    \begin{align*}
        \sum_{i=1}^{\cnk}\log(\det(Z_i)) &=  \sum_{i=1}^{\cnk} \sum_{J \in S_i}\log(\det(Y_J)) \\ 
        &= \sum_{J \in \mathcal{J}} \log(\det(Y_J)) = -f^{\mathrm{FW}(k)}(\mathcal{Y}),
    \end{align*}
    completing the proof.
\end{proof}

We continue to prove Lemma \ref{lemma1}. In the proof we use Minkowski's determinant inequality, which we restate for convenience.

\begin{theorem}(Minkowski's determinant inequality, see, e.g. \cite[Theorem 4.1.8]{Marcus64})
    Let $A,B \in \mathbb{S}^n_+$. Then
    \begin{equation}\label{detinequ}
        \left(\det(A+B)\right)^{\frac{1}{n}} \ge \det(A)^\frac{1}{n}+\det(B)^\frac{1}{n}.
    \end{equation}
\end{theorem}

\begin{proof}{(Lemma \ref{lemma1})}
    The self-concordance of $f^{\mathrm{FW}(k)}$ on $\mathrm{int}\left(\mathbb{S}^{(n,k)_+} \right)$ follows immediately from the self-concordance of $-\log \det(X)$ on $\mathrm{int}\left(\mathbb{S}^n\right)$. By assumption $X = \Psi(\mathcal{Y}) = \sum_{i=1}^{\cnk}Z_i \in \mathrm{FW}_n(k)$. Therefore,
    \begin{equation*}
            \frac{1}{\cnk}\det(X)^{1/n}  \ge \frac{1}{\cnk}\sum_{i=1}^{\cnk}\det(Z_i)^{1/n},
    \end{equation*}
    where the inequality follows from Minkowski's determinant inequality \eqref{detinequ}. Applying the logarithm on both sides and rearranging the left-hand-side yields
    \begin{align*}
        \frac{1}{n}\log(\det(X))-\log\left(\cnk\right) &\ge \log \left( \frac{1}{\cnk}\sum_{i=1}^{\cnk}\det(Z_i)^{1/n}\right).
    \end{align*}
    Using the fact that the logarithm is concave we see
    \begin{align*}
        \frac{1}{n}\log(\det(X))-\log(\cnk) &\ge  \frac{1}{\cnk}\sum_{i=1}^{\cnk}\frac{1}{n} \log \left(\det(Z_i)\right).
    \end{align*}
    Multiplying by $n\cnk$ leads to
    \begin{align*}
        -\cnk\left(f^{\mathrm{SDP}}(X)+n \log\left(\cnk\right)\right) &=\cnk\left(\log(\det(X))-n \log\left(\cnk\right)\right) \\
        &\ge \sum_{i=1}^{\cnk}\log(\det(Z_i))= -f^{\mathrm{FW}(k)}(\mathcal{Y}).
    \end{align*}
\end{proof}

The following corollary is analogous to Corollary 4.5 from \cite{Sznaier22}.
\begin{corollary}\label{corr1iterativeScheme}
    If
    \[
        \mathcal{Y}_0 = \{ Y_J = 1/\cnk I_{k \times k} : J \subset [n], \vert J \vert = k \}
    \]
    then $X=\Psi(\mathcal{Y}_0)=I$ and
    \begin{align*}
        f^{\mathrm{FW}(k)}(\mathcal{Y}_0)& =\cnk f^{\mathrm{SDP}}(X)+n \cnk\log\left(\cnk\right) \\
        &=n \cnk\log\left(\cnk\right).
    \end{align*}
\end{corollary}
\begin{proof}
    The first statement follows when noting that each $i \in [n]$ lies in exactly $\binom{n-1}{k-1}$ subsets of $[n]$ of size $k$. The reason is that when fixing $i$, there are $n-1$ elements left out of which we want to choose $k-1$ more elements to make a set of size $k$. For the second statement note that
    \[
        \log\left( \det \left( \frac{1}{\cnk}I_{k \times k}\right) \right) =\log \left(\left(\cnk\right)^{-k} \right) = -k \log\left(\cnk\right).
    \]
    The result follows when noting that $k \binom{n}{k} = n \cnk$.
\end{proof}

\section{Relations of the barrier functions}\label{sec:relating_barriers}
To prove convergence of our algorithm we need two essential ingredients. First, we need to prove that the predictor step reduces the current objective value sufficiently, and secondly, we must prove that the corrector step converges to a point close to the central path. Moreover, we have to show that our criterion to decide which subroutine to call is valid. The issue here is that we compute the Newton decrement of $f^{\mathrm{FW}(k)}$ at $\mathcal{Y}_0$, but we need to be able to assert that the Newton decrement of $f^{\mathrm{SDP}}$ at $X_\ell$ is small enough.

The next result we present will allow us to lower bound the progress made by the corrector step. For this we need to be able to compare the barrier functions for $\mathbb{S}^n_+$ and $\mathbb{S}^{(n,k)}_+$.
We assume we have a given feasible solution $X_\ell$ such that $\langle A^{(\ell)}_0, I \rangle = v$. Define the vector $b(v) := (v, b_1, \ldots, b_m)^T$. For further reference, consider
\begin{equation}\label{SDPipm}  
        \mathrm{min} \; \left\{ f^{\mathrm{SDP}}(X) :  \langle A^{(\ell)}_i, X \rangle = b(v)_i \; \forall i = 0, 1, \ldots,m,  X \in \mathbb{S}^n_+\right\},
\end{equation}

which we would like to compare to

\begin{equation}\label{reducedSDP}
       \min \; \left\{ f^{\mathrm{FW}(k)}(\mathcal{Y}) : \mathcal{Y} \in L_{\ell}^{\Psi}(v) \cap \mathbb{S}^{(n,k)}_+ \right\}.
\end{equation}
Suppose $\mathcal{Y}^\ast$ is an approximate solution to \eqref{reducedSDP}. Defining
\[
    X_{\ell+1} = X_{\ell}^{1/2} \Psi(\mathcal{Y}^\ast) X_{\ell}^{1/2},
\]
we find that $X_{\ell} \in \mathcal{F}_{\mathrm{SDP}}$ for all $\ell$. In other words, the points $X_\ell$ we obtain via this procedure are all feasible for the original SDP \eqref{IterativeSchemeprimal}. The following lemma allows us to lower bound the decrease achieved by one corrector step in terms of an element in $\mathbb{S}^{(n,k)}_+$.

\begin{lemma}\label{lemma:comp_values}
    Let $\mathcal{Y}^\ast$ be a feasible solution to \eqref{reducedSDP} and
     $\mathcal{Y}_0$ as in \eqref{Y0}. Further, let  $X_{\ell+1} = X_{\ell}^{1/2}\Psi(\mathcal{Y}^\ast)X_{\ell}^{1/2}$ for $X_\ell$ a feasible solution. Then
    \[
        \cnk \left( f^{\mathrm{SDP}}(X_{\ell})-f^{\mathrm{SDP}}(X_{\ell+1}) \right) \ge f^{\mathrm{FW}(k)}(\mathcal{Y}_{0})-f^{\mathrm{FW}(k)}(\mathcal{Y}^\ast).
    \]
\end{lemma}
\begin{proof}
    The proof follows immediately when noting that
    \begin{align*}
        \cnk \Big( f^{\mathrm{SDP}}(X_{\ell})-&f^{\mathrm{SDP}}(X_{\ell+1}))  \Big)  = \cnk \left( f^{\mathrm{SDP}}(X_{\ell}) -f^{\mathrm{SDP}}( X_{\ell}^{1/2} \Psi(\mathcal{Y}^\ast)X_{\ell}^{1/2} \right) \\
         &= \underbrace{n \cnk \log(\cnk)}_{= f^{\mathrm{FW}(k)}(\mathcal{Y}_0) \text{ by Cor. \ref{corr1iterativeScheme}}} \underbrace{-f^{\mathrm{SDP}}(\Psi(\mathcal{Y}^\ast))-n \cnk \log \cnk}_{\ge -f^{\mathrm{FW}(k)}(\mathcal{Y}^\ast) \text{ by Lemma \ref{lemma1}}}
    \end{align*}
\end{proof}

\subsection{Relation of the Newton decrements}

In this subsection we will prove that we can upper bound the Newton decrement of $f^{\mathrm{SDP}}$ at the identity in terms of the Newton decrement of $f^{\mathrm{FW}(k)}$ at $\mathcal{Y}_0$. We now define the following operator

\[
    \Psi^{\dagger} : \mathbb{S}^{n} \rightarrow \mathbb{S}^{(n,k)}
\]
via
\[
    \left( \Psi^{\dagger}(X) \right)_J = \left( \frac{1}{\cnk} I +\frac{1}{\dnk}(ee^T-I)\right) \circ X_{J,J} \quad \text{ for } J \subset [n], \vert J \vert = k,
\]
where $\circ$ denotes the Hadamard product. See Figure \ref*{fig:surj} for a visualization of the surjection from $\mathbb{S}^{(n,k)}_+$ to $\mathrm{FW}_n(k)$.

\begin{figure}[h!]
    \includegraphics[width=0.9\textwidth]{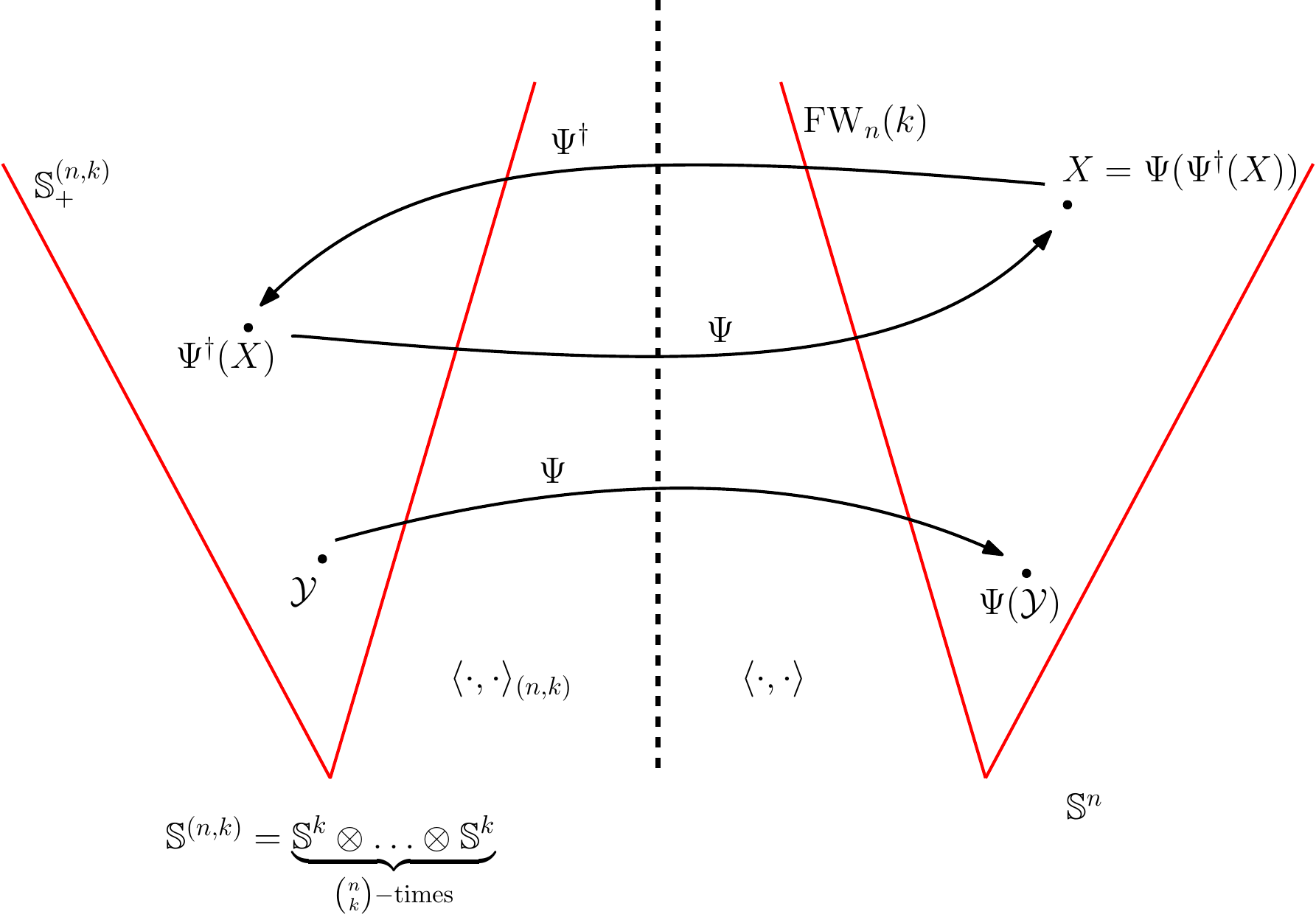}
    \caption{Visualization the surjection from $\mathbb{S}^{(n,k)}_+$ to $\mathrm{FW}_n(k)$}\label{fig:surj}
\end{figure}

This operator satisfies

\[
    \Psi(\Psi^{\dagger}(X)) = X \text{ for all } X \in \mathbb{S}^n.
\]

An inner product on $\mathbb{S}^{(n,k)}$ given by
\[
    \langle \mathcal{X}, \mathcal{Y} \rangle_{(n,k)} := \sum_{\vert J \vert =k} \langle X_J, Y_J \rangle,
\]
and it is well-defined for $\mathcal{X} = \{ X_J \in \mathbb{S}^k : \vert J \vert = k\}, \mathcal{Y} = \{ Y_J \in \mathbb{S}^k : \vert J \vert = k\}$.
It is straightforward to verify the following relation between the two norms.
\begin{lemma}\label{lem:norm}
    For any $X \in \mathbb{S}^n$ we have
    \[
        \vert \vert \Psi^\dagger(X) \vert \vert_{(n,k)} \le \vert \vert X \vert \vert. 
    \]
\end{lemma}

Suppose now $X_\ell$ is a feasible solution to \eqref{relaxell} such that $\langle A_0, X_\ell \rangle = v$. We define the vector $b(v) := (v, b_1, \ldots, b_m)^T$ as well as the two subspaces
\begin{equation*}
    L^{\Psi}_{\ell}  = \{ \mathcal{Y} \in \mathbb{S}^{(n,k)} : (\mathcal{A}^{(\ell)} \circ \Psi) (\mathcal{Y}) = b\}
\end{equation*}
and
\begin{equation*}
    L_\ell = \{ X \in \mathbb{S}^n : \mathcal{A}^{(\ell)}_0(X) = b \}.
\end{equation*}
 Note that we may also add an equality for the objective, in which case we will refer to the following operator
 \[
\mathcal{A}_0^{(\ell)}(X) = (\langle A^{(\ell)}_0, X\rangle ,\langle A^{(\ell)}_1,X\rangle ,\ldots,\langle A^{(\ell)}_m, X\rangle) \in \mathbb{R}^{m+1}.
\]
The respective subspaces will be denoted as follows
\begin{equation}\label{Lpsiellv}
    L^{\Psi}_{\ell}(v)  = \{ \mathcal{Y} \in \mathbb{S}^{(n,k)} : (\mathcal{A}_0^{(\ell)} \circ \Psi) (\mathcal{Y}) = b(v)\}
\end{equation}
and
\begin{equation}\label{Lellv}
    L_\ell(v) = \{ X \in \mathbb{S}^n : \mathcal{A}^{(\ell)}_0(X) = b(v)\}.
\end{equation}
When we consider the subspaces defined via the operator with respect to the initial data matrices, we omit the subscript $\ell$, e.g.,
\[
    L^{\Psi} = \{ \mathcal{Y} \in \mathbb{S}^{(n,k)} : \langle A_i, \Psi(\mathcal{Y}) \rangle = b_i \, , \, \forall i \in [m] \}.
\]

The following lemma corresponds to Lemma A.2 in \cite{Sznaier22}, and allows us to bound the Newton decrement of $f^{\mathrm{SDP}}_{\vert L}$ in terms of $f^{\mathrm{FW}(k)}_{\vert L}$.

\begin{lemma}\label{lemma:comp_decrements}
Assume $\mathcal{Y}_0 \in L^{\Psi}$ and $I \in L$.
At $\mathcal{Y}_0$ one has
\[
\Delta \left({f^{\mathrm{FW}(k)}_{\vert _{L^{\Psi}}}},\mathcal{Y}_0\right) \ge \frac{\Delta\left(\cnk f^{\mathrm{SDP}}_{\vert _L},I\right)}{\sqrt{\cnk}} = \sqrt{\cnk}\Delta\left(f^{\mathrm{SDP}}_{\vert_L},I\right).
\]
\end{lemma}
\begin{proof}
    Following \eqref{varNorm} we have

    \[
        \Delta\left(f^{\mathrm{FW}(k)}_{\vert_{L^\Psi}}, \mathcal{Y}\right) \ge \frac{\langle d, n^{\mathrm{FW}}(\mathcal{Y}) \rangle_{(n,k),\mathcal{Y}}}{\vert \vert d \vert \vert _{(n,k),\mathcal{Y}}} \text{ for all }  d \in L \setminus \{ 0\}.
    \]

    Choosing $d = \Psi^{\dagger}(n^{\mathrm{SDP}}_L(X)) \in L$ leads to
    \begin{align*}
        \Delta\left(f^{\mathrm{FW}(k)}_{\vert _{L^\Psi}}, \mathcal{Y}\right) & \ge \frac{\langle \Psi^\dagger(n^{\mathrm{SDP}}_L(X)),n^{\mathrm{FW}}(\mathcal{Y}) \rangle_{{(n,k),\mathcal{Y}}}}{\vert \vert \Psi^\dagger(n^{\mathrm{SDP}}_L(X))\vert \vert_{{(n,k),\mathcal{Y}}}},
    \end{align*}
    and evaluating the expression at $\mathcal{Y}_0$ we find
    \begin{align*}
        \Delta\left(f^{\mathrm{FW}(k)}_{\vert_{L^\Psi}}, \mathcal{Y}_0\right) & \ge \frac{\langle \Psi^\dagger(n^{\mathrm{SDP}}_L(X)),n^{\mathrm{FW}}(\mathcal{Y}_0) \rangle_{{(n,k),\mathcal{Y}}}}{\vert \vert \Psi^\dagger(n^{\mathrm{SDP}}_L(X))\vert \vert_{{(n,k),\mathcal{Y}}}} \\
        &= \frac{ \langle \Psi^\dagger (n^{\mathrm{SDP}}_L(X))  , -g^{\mathrm{FW}}(\mathcal{Y}_0) \rangle_{(n,k)}}{ \cnk \vert \vert\Psi^\dagger (n^{\mathrm{SDP}}_L(X)) \vert \vert_{(n,k)}} \\ 
        & \ge \frac{ \langle \Psi^\dagger (n^{\mathrm{SDP}}_L(X))  , (I, I , \ldots, I) \rangle_{(n,k)}}{ \vert \vert n^{\mathrm{SDP}}_L(X) \vert \vert} \\ 
        & =   \frac{\mathrm{tr}(n_L^{\mathrm{SDP}}(X))}{\vert \vert n_L^{\mathrm{SDP}}(X)\vert \vert},
    \end{align*}
    where the second inequality follows from Lemma \ref{lem:norm}.
    Setting $X = I$ and noting
    \begin{align*}
        \mathrm{tr}(n_L^{\mathrm{SDP}}(I)) = \langle I, n_L^{\mathrm{SDP}}(I) \rangle &= \frac{1}{\cnk}\langle g^{\mathrm{SDP}}(I), -n_L^{\mathrm{SDP}}(I)  \rangle \\
        &= \frac{1}{\cnk}\left(\Delta\left(\cnk f^{\mathrm{SDP}}_{\vert_L}, I \right)\right)^2
    \end{align*}
        we conclude
    \begin{align*}
         \Delta\left(f^{\mathrm{FW}(k)}_{\vert_{L^\Psi}}, \mathcal{Y}_0\right) \ge \frac{1}{\cnk}\frac{\Delta\left(\cnk f^{\mathrm{SDP}}_{\vert_L},I\right)^2}{\vert \vert n_L^{\mathrm{SDP}}(I) \vert \vert} = \frac{\Delta\left(\cnk f^{\mathrm{SDP}}_{\vert_L},I\right)}{\sqrt{\cnk}},
    \end{align*}
    because
    \[
        \vert \vert n_L^{\mathrm{SDP}}(I) \vert \vert = \frac{\Delta\left(\cnk f_{\vert L}^{\mathrm{SDP}}, I\right)}{\sqrt{\cnk}} = \sqrt{\cnk } \Delta\left(f_{\vert L}^{\mathrm{SDP}}, I\right).
    \]
\end{proof}

\section{Complexity analysis}\label{sec:complexity}

We begin the complexity analysis with the following lemma, which helps us to check whether the current point is close enough to the central path of the SDP. 

\begin{lemma}\label{lemma_closeEnough}
    Let $X_{\ell}$ be a feasible iterate for the SDP \eqref{SDPipm} and let the objective value at $X_{\ell}$ be $v$. Define the two subspaces $L^{\Psi}_{\ell}(v)$, $L_\ell$ as in \eqref{Lpsiellv}, \eqref{Lell} respectively. Then, if
    \[
        \Delta\left(f^{\mathrm{FW}(k)}_{\vert L^{\Psi}_{\ell}(v)}, \mathcal{Y}_0\right) \le \frac{1}{14},
    \]
    one has
    \[
        \Delta\left({f^{\mathrm{SDP}}_{{\eta_{v}}\vert_{L_\ell}}}, I\right) \le \frac{1}{9},
    \]
    where 
    \[
        f^{\mathrm{SDP}}_{\eta_{v}}(X) = \eta_v \langle A_0, X \rangle - \log \det (X),
    \]
    and $\eta_v$ is such that
    \[
        v = \min_{X \in L_\ell} f^{\mathrm{SDP}}_{\eta_{v}}(X).
    \] 
\end{lemma}
\begin{proof}
    By Lemma \ref{lemma:comp_decrements} we know that
    \[
       \frac{1}{14} \ge \Delta\left(f^{\mathrm{FW}(k)}_{\vert L^{\Psi}_{\ell}(v)}, \mathcal{Y}_0\right) \ge \Delta\left(f^{\mathrm{SDP}}_{\vert L_{\ell}(v)}, I\right).
    \]
    Let now $z(v)$ be the point on the central path of the rotated SDP with objective value $v$ and let the corresponding parameter be $\eta_v$. By Theorem 2.2.5 from \cite{Renegar2001} we have
    \begin{equation}\label{rel:oneovereleven}
        \vert\vert z(v)-I \vert\vert_I \le \Delta\left(f^{\mathrm{SDP}}_{\vert L_{\ell}(v)}, I\right) +\frac{3 \Delta\left(f^{\mathrm{SDP}}_{\vert L_{\ell}(v)}, I\right)^2}{\left( 1- \Delta\left(f^{\mathrm{SDP}}_{\vert L_{\ell}(v)}, I\right) \right)^3} \le \frac{1}{11}.
    \end{equation}
    Let $X_+$ be the point returned by taking a Newton step at $X = I$ with respect to the function $f^{\mathrm{SDP}}_{\eta_v}$ restricted to $L_\ell$. By Theorem 2.2.3 in \cite{Renegar2001} we have
    \[
         \frac{\vert\vert z(v)-I\vert\vert^2_I}{1-\vert\vert z(v)-I\vert\vert_I} \ge \vert\vert X_+ -z(v)\vert\vert_I
    \]
    and hence
    \begin{align*}
        \Delta\left(f^{\mathrm{SDP}}_{\eta_{v}\vert_{L_\ell}}, I\right) = \vert\vert X_+-I\vert\vert_I  & \le \vert\vert X_+ -z(v)\vert\vert_I + \vert\vert z(v)- I\vert\vert_I \\
        & \le \frac{\vert\vert z(v)-I\vert\vert^2_I}{1-\vert\vert z(v)-I\vert\vert_I} + \vert\vert z(v)- I\vert\vert_I \le \frac{1}{9}.
    \end{align*}
\end{proof}

The Newton decrement of the rotated SDP being smaller than $1/9$ means we can safely perform the next predictor step. If the current point is too far away from the central path and one were to perform the predictor step the direction may not be approximately tangential to the central path. Hence, once the Newton decrement of the factor width program is small enough, so is the one of the SDP and we can perform the next predictor step, knowing the direction will be approximately tangential to the central path. After each predictor step we may have to take several corrector steps, to get back close to the central path.

\subsection*{Corrector step}
We will now find an upper bound on the number of corrector steps needed to get close to the central path. We know from Lemma \ref*{lemma:comp_values} that a decrease in the barrier for the factor width cone will lead to a decrease in the barrier function for our original SDP, meaning we made progress towards its central path. The following lemma asserts that if we are too far away from the central path we can attain at least a constant reduction in the barrier of the factor width cone and therefore obtaining a constant reduction in the SDP barrier as well.
\begin{lemma}\label{lemma_decrease}
    Let $X_{\ell}$ be a feasible iterate for the SDP \eqref{SDPipm} and let the objective value at $X_{\ell}$ be $v$. Define the subspace $L^{\Psi}_{\ell}(v)$ as in \eqref{Lpsiellv}. If
    \[
        \Delta\left(f^{\mathrm{FW}(k)}_{\vert L^{\Psi}_{\ell}(v)}, \mathcal{Y}_0\right) > \frac{1}{14}
    \]
    then
    \[
        f^{\mathrm{FW}(k)}_{\vert L^{\Psi}_{\ell}(v)}(\mathcal{Y}_0)-f^{\mathrm{FW}(k)}_{\vert L^{\Psi}_{\ell}(v)}(\mathcal{Y}^{\ast}) \ge \frac{1}{2688}.
    \]
\end{lemma}
\begin{proof}
    If $\Delta\left(f^{\mathrm{FW}(k)}_{ \vert L^{\Psi}_{\ell}(v)}, \mathcal{Y}_0\right) > \frac{1}{14}$ the corrector step will employ a line search to find $\mathcal{Y}^{\ast}$, i.e. the point in $L^{\Psi}_{\ell}(v)$ that minimizes $f^{\mathrm{FW}(k)}$. Let $n_{L^{\Psi}_{\ell}(v)}(\mathcal{Y}_0)$ be the Newton step taken from $\mathcal{Y}_0$ and let $t = \frac{1}{8 \| n_{L^{\Psi}_{\ell}(v)}(\mathcal{Y}_0) \|_{(n,k),\mathcal{Y}_0}}$, where the norm in the denominator is the local norm at $\mathcal{Y}_0$ induced by $\langle \cdot, \cdot \rangle_{(n.k)}$. Then, for
    \[
        \tilde{\mathcal{Y}} = \mathcal{Y}_0 + t \, n_{L^{\Psi}_{\ell}(v)}(\mathcal{Y}_0)
    \]
    we find by Theorem 2.2.2 in \cite{Renegar2001}
    \begin{align*}
        f^{\mathrm{FW}(k)}(\tilde{\mathcal{Y}}) & \le f^{\mathrm{FW}(k)}(\mathcal{Y}_0)-\frac{1}{14}\frac{1}{8}+\frac{1}{2}\left( \frac{1}{8} \right)^2+ \frac{(1/8)^3}{3(1-1/8)} \\
        &\le f^{\mathrm{FW}(k)}(\mathcal{Y}_0)-\frac{1}{2688}.
    \end{align*}
\end{proof}
Note that this implies together with Lemma \ref{lemma:comp_values} that
\begin{equation}
    \begin{aligned}
        \frac{1}{2688} \le f^{\mathrm{FW}(k)}(\mathcal{Y}_0)-f^{\mathrm{FW}(k)}(\tilde{\mathcal{Y}}) &\le  f^{\mathrm{FW}(k)}(\mathcal{Y}_0)- f^{\mathrm{FW}(k)}(\mathcal{Y}^\ast) \\
        & \le \cnk \left( f^{\mathrm{SDP}}(X_\ell)-f^{\mathrm{SDP}}(X_{\ell+1}) \right).
    \end{aligned}
\end{equation}
Knowing each line search reduces the distance to the targeted point on the central path at least by a constant amount will allow us to bound the number of line searches we need to get close enough if we have an upper bound on the distance of the result of the predictor step and the corresponding point on the central path of the SDP.

\begin{lemma}\label{lemma:distance}
    Let $X_1$ be close to a point $z(v_1)$ on the central path of the SDP in the sense that $\Delta\left(f^{\mathrm{SDP}}_{L_{\ell}(v_1)}, X_1\right) \le \frac{1}{9}$. Further, let $X_2$ be the result of the predictor step and $z(v_2)$ be the point on the central path with the same objective value as $X_2$. Then
    \[
        f^{\mathrm{SDP}}(X_{2})-f^{\mathrm{SDP}}(z(v_2)) \le n \left( \log \frac{1}{1-\sigma}\right)+\frac{1}{154}.
    \]
\end{lemma}
\begin{proof}
    A proof of this statement for generic self-concordant barriers may be found on page 54 of \cite{Renegar2001}. We have used that the barrier parameter for the barrier of the psd cone is given by $\vartheta_{f^{\mathrm{SDP}}} = n$. \qedhere
\end{proof}

\begin{lemma}
    Let $v_2$ be the objective value of the result $X_2$ of the predictor step. The maximum number $K$ of line searches needed to find a point $X_{K+2}$ which is close enough to $z(v_2)$ in the sense that $\Delta\left(f^{\mathrm{SDP}}_{\vert L_{\ell}(v_2)}, X_{K+2}\right) \le \frac{1}{9}$ is
    \[
        K = \left\lceil 2688 \cnk \left( n \log\left(\frac{1}{1-\sigma}\right) +\frac{1}{154} \right)\right\rceil ,
    \]
    where $z(v_2)$ is the point on the central path with objective value $v_2$.
\end{lemma}
\begin{proof}
    We know that the distance between the result of the predictor phase and the targeted point on the central path is at most $n \left( \log \frac{1}{1-\sigma}\right)+\frac{1}{154}$ by Lemma \ref{lemma:distance}. Moreover, using Lemma \ref{lemma_decrease} we find that in each corrector step we reduce this distance by at least $\frac{1}{2688\cnk}$, unless the SDP Newton decrement at $I$ is already small enough to perform the next predictor step. If after rescaling the Newton decrement of the factor width program satisfies
    \[
        \Delta\left(f^{\mathrm{FW}(k)}_{\vert L^{\Psi}_{\ell}(v)}, \mathcal{Y}_0\right) > \frac{1}{14},
    \]
    thereby implying by Lemma \ref{lemma_closeEnough} that $I$ is not close to the central path of the SDP we can perform another corrector step yielding at least a constant decrease of $\frac{1}{2688\cnk}$ of the distance to the central path, and rescale again. This process can be continued until we do not get such a constant decrease anymore at which point we know we must be close enough to the central path, in the sense of Lemma \ref{lemma_closeEnough}. This is because if the decrease is not greater than $\frac{1}{2688\cnk}$ we know that the Newton decrement cannot satisfy
    \[
        \Delta\left(f^{\mathrm{FW}(k)}_{\vert L^{\Psi}_{\ell}(v)}, \mathcal{Y}_0\right) > \frac{1}{14},
    \]
    from which follows by  Lemma \ref{lemma_closeEnough} that
    \[
        \Delta\left(f^{\mathrm{SDP}}_{L_{\ell}(v)}, I\right) \le \frac{1}{9}.
    \]
    This implies we are close enough to the central path to perform the next predictor step.
    Hence, after at most
    \[
         K = \left \lceil 2688 \cnk \left( n \log\left(\frac{1}{1-\sigma}\right) +\frac{1}{154} \right) \right \rceil
    \]
    corrector steps we are close enough to the central path so that we can perform the next predictor step.
\end{proof}

\subsection*{Predictor step}

We will make use of the analysis of the short step interior point method discussed in Section 2.4.2 in \cite{Renegar2001}. We will show that each predictor step reduces the objective value by an amount at least as large as the objective decrease by the short-step interior point method. This will allow us to conclude the maximum number of predictor steps needed to obtain an $\varepsilon$ optimal solution of the given SDP. Note that the decrease in objective value obtained by our predictor method is as follows. Let $X$ be the point from where the predictor method starts and $-(A_0)_X := -H(X)A_0$ be the direction. Then for $\sigma \ge \frac{1}{4}$ we find
\begin{align*}
    \langle A_0, X-s^\ast \sigma \, (A_0)_X  \rangle &= \langle (A_0)_X, X \rangle - s^\ast \sigma \langle A_0, (A_0)_X \rangle \\
    & \le \langle A_0, X \rangle -\frac{1}{4}\| (A_0)_X \|_X.
\end{align*}
This implies the decrease is at least as large the one obtained in one iteration of the short-step method, as discussed in \cite[§~2.4.2]{Renegar2001}. Renegar's analysis shows that short-step method leads to an $\varepsilon$ optimal solution in at most 
\[
    K = 10\sqrt{\vartheta_f}\log(\vartheta_f /(\varepsilon \, \eta_0))
\] 
steps, where $\eta_0$ is such that our starting point $X_0$ is close to $z_{\eta_0}$. By an $\varepsilon$ optimal solution we mean a feasible solution $X$ such that 
\[
    v^\ast_\mathrm{SDP} \le \langle A_0 , X \rangle \le v^\ast_\mathrm{SDP} + \varepsilon.
\]

\subsection*{Predictor and corrector steps combined}
Combining the complexity analysis of predictor and corrector steps we arrive at the following theorem.
\begin{theorem}
    Let $X_0$ be a feasible solution of the SDP \eqref{IterativeSchemeprimal} and assume it is close to some point $z_{\eta_0}$ on the corresponding central path in the sense that $\Delta\left(f_{\vert_{L(v)}}^{\mathrm{SDP}},X_0\right)<1/14$, where $L$ is as in \eqref{Lellv} for $v = \langle A_0, X_0 \rangle$. Algorithm \ref{alg:pc_ipm} converges to an $\varepsilon$ optimal solution in at most
    \begin{align*}
        K &= \left \lceil 2688 \cnk \left( n \log\left(\frac{1}{1-\sigma}\right) +\frac{1}{154} \right) \right \rceil  10\sqrt{n}\log(n /(\varepsilon \, \eta_0)) \\
		& = O\left({n-1 \choose k-1}n^{3/2}\log\left(\frac{1}{1-\sigma}\right)\log\left(\frac{n}{\epsilon \eta_0}\right)\right).
	\end{align*}
    steps.
\end{theorem}

The assumption of a starting point "close to the central path" may be satisfied by the self-dual embedding strategy \cite{deKlerk1997}. Alternatively, one may first solve an auxiliary SDP problem, as in \cite[§~2.4.2]{Renegar2001}, by using the algorithm we have presented. The solution of this auxiliary problem then yields a point close to the central path of the original SDP problem.

\section{Discussion and future prospects}\label{sec:discussion}
We finish with a brief discussion on the prospects of efficient implementation of Algorithm \ref{alg:pc_ipm}.
\subsection*{Parallelization}
Essentially, the contribution of the present paper lies in providing an algorithm for solving SDPs which is much more suitable for parallelization than the ordinary interior point method working over $\mathbb{S}^n_+$. Given common memory access, the computation of the necessary data for the respective cone factors $\mathbb{S}^k_+$ is local, meaning these tasks can be distributed among processor cores leading to a runtime decrease since each corrector step involves ${{n}\choose{k}}$ parallel computations of $O(k^3m + k^2m^2 + m^3)$ flops. This offers the potential to perform the centering steps much more quickly than for SDP interior point methods through parallel computation.

\subsection*{Replacing the predictor step}
In their paper \cite{Sznaier22}, the authors propose to perform a fixed number of decrease steps, where a decrease step consists of solving \eqref{fwProg} and rescaling with respect to the optimal solution. In our algorithm we considered a different method to decrease the objective value, i.e., the predictor method, where we use the traditional SDP affine scaling direction.

\subsection*{Tractability of factor width cones}
The entire approach described in this paper relies on the premise that one may optimize more efficiently over $\mathrm{FW}_n(k)$ than over $\mathbb{S}^n_+$. In practice this has not yet been demonstrated convincingly for $k > 2$, although the consensus is that it should be possible. Some recent ideas that could be useful in this regard are:
\begin{itemize}
    \item the idea to optimize over the dual cone of $\mathrm{FW}_n(k)$ by utilizing clique trees \cite{Zhang2021}
    \item a variation on the factor width cone involving fewer blocks \cite{Zheng2022}.
\end{itemize}
In addition, it would be very helpful to know a computable self-concordant barrier functional for the cone $\mathrm{FW}_n(k)$, as well as its complexity parameter.
\\\\
\textbf{Acknowledgements.} The authors would like to thank Georgina Hall for insightful discussions on the topic on multiple occasions. Moreover, the authors thank Micha\"el Gabay and Arefeh Kavand for fruitful conversations on different angles of the subject matter. 
\\\\
\textbf{Funding.} 
This work is supported by the European Union’s Framework Programme for Research and Innovation Horizon 2020 under the Marie Skłodowska-Curie grant agreement N. 813211 (POEMA).


\begin{thebibliography}{10}

    \bibitem{Ahmadi2015}
    Amir~Ali Ahmadi, Sanjeeb Dash, and Georgina Hall.
    \newblock Optimization over structured subsets of positive semidefinite
      matrices via column generation.
    \newblock {\em Discrete Optimization}, 24:129--151, 2017.
    \newblock Conic Discrete Optimization.
    
    \bibitem{Ahmadi2017}
    {Amir Ali} Ahmadi and Georgina Hall.
    \newblock {\em Sum of squares basis pursuit with linear and second order cone
      programming}, pages 27--53.
    \newblock Contemporary Mathematics. American Mathematical Society, United
      States, 2017.
    
    \bibitem{Ahmadi2014}
    Amir~Ali Ahmadi and Anirudha Majumdar.
    \newblock {DSOS} and {SDSOS} optimization: {LP} and {SOCP}-based alternatives
      to sum of squares optimization.
    \newblock {\em 2014 48th Annual Conference on Information Sciences and Systems
      (CISS)}, pages 1--5, 2014.
    
    \bibitem{Ahmadi2019}
    Amir~Ali Ahmadi and Anirudha Majumdar.
    \newblock {DSOS} and {SDSOS} optimization: More tractable alternatives to sum
      of squares and semidefinite optimization.
    \newblock {\em {SIAM} Journal on Applied Algebra and Geometry}, 3(2):193--230,
      jan 2019.
    
    \bibitem{Alizadeh1998}
    Farid Alizadeh, Jean-Pierre~A. Haeberly, and Michael~L. Overton.
    \newblock Primal-dual interior-point methods for semidefinite programming:
      Convergence rates, stability and numerical results.
    \newblock {\em SIAM Journal on Optimization}, 8(3):746--768, 1998.
    
    \bibitem{Baranyai1975}
    Zsolt Baranyai.
    \newblock On the factorization of the complete uniform hypergraph.
    \newblock {\em Infinite and Finite Sets}, 1:91--108, 1975.
    \newblock Proceedings of a Colloquium held at Keszthely, June 25-July 1, 1973.
      Dedicated to Paul Erdős on his 60th Birthday.
    
    \bibitem{Blekherman2022}
    Grigoriy Blekherman, Santanu~S. Dey, Marco Molinaro, and Shengding Sun.
    \newblock Sparse psd approximation of the psd cone.
    \newblock {\em Mathematical Programming}, 191(2):981--1004, 2022.
    
    \bibitem{Boman2005}
    Erik~G. Boman, Doron Chen, Ojas Parekh, and Sivan Toledo.
    \newblock On factor width and symmetric {H}-matrices.
    \newblock {\em Linear Algebra and its Applications}, 405:239--248, aug 2005.
    
    \bibitem{Borchers1999}
    Brian Borchers.
    \newblock {CSDP}, {A} {C} library for semidefinite programming.
    \newblock {\em Optimization {M}ethods and {S}oftware}, 11(1-4):613--623, 1999.
    
    \bibitem{deKlerk1997}
    E.~{de Klerk}, C.~Roos, and T.~Terlaky.
    \newblock Initialization in semidefinite programming via a self-dual
      skew-symmetric embedding.
    \newblock {\em Operations Research Letters}, 20(5):213--221, 1997.
    
    \bibitem{Klerk2016}
    Etienne de~Klerk and Frank Vallentin.
    \newblock On the {T}uring model complexity of interior point methods for
      semidefinite programming.
    \newblock {\em SIAM Journal on Optimization}, 26(3):1944--1961, 2016.
    
    \bibitem{Fujisawa2002}
    Katsuki Fujisawa, Masakazu Kojima, Kazuhide Nakata, and Makoto Yamashita.
    \newblock {SDPA} (semidefinite programming algorithm) user's manual — version
      6.00.
    \newblock {\em Math. Comp. Sci. Series B: Oper. Res.}, 12 2002.
    
    \bibitem{Lasserre2001}
    Jean~B. Lasserre.
    \newblock Global optimization with polynomials and the problem of moments.
    \newblock {\em SIAM Journal on Optimization}, 11(3):796--817, 2001.
    
    \bibitem{Marcus64}
    Marvin Marcus and Henryk Minc.
    \newblock {\em A survey of matrix theory and matrix inequalities}, volume~14.
    \newblock Allyn and Bacon, Inc., 1964.
    \newblock p. 115.
    
    \bibitem{mosek}
    {MOSEK, ApS}.
    \newblock {MOSEK} {O}ptimization {Software}.
    \newblock Technical report, Version 9.1.9, 2019.
    \newblock http://docs.mosek.com/9.1/toolbox/index.html.
    
    \bibitem{Nesterov1994}
    Yurii Nesterov and Arkadii Nemirovskii.
    \newblock {\em Interior-Point Polynomial Algorithms in Convex Programming}.
    \newblock Society for Industrial and Applied Mathematics, 1994.
    
    \bibitem{Renegar2001}
    J.~Renegar.
    \newblock {\em A Mathematical View of Interior-Point Methods in Convex
      Optimization}.
    \newblock MPS-SIAM Series on Optimization. Society for Industrial and Applied
      Mathematics, 2001.
    
    \bibitem{Sznaier22}
    Biel Roig-Solvas and Mario Sznaier.
    \newblock A globally convergent {LP} and {SOCP}-based algorithm for
      semidefinite programming.
    \newblock 2022.
    \newblock preprint, \url{https://arxiv.org/abs/2202.12374}.
    
    \bibitem{sedumi}
    Jos~F Sturm.
    \newblock Using {S}e{D}u{M}i 1.02, a {MATLAB} toolbox for optimization over
      symmetric cones.
    \newblock {\em Optimization {M}ethods and {S}oftware}, 11(1-4):625--653, 1999.
    
    \bibitem{Toh1999}
    K.C. Toh, M.J. Todd, and R.H. Tütüncü.
    \newblock {SDPT}3 - a {MATLAB} software package for semidefinite programming,
      version 1.3.
    \newblock {\em Optimization Methods and Software}, 11(1):545--581, 1999.
    
    \bibitem{SDPA}
    Makoto Yamashita, Katsuki Fujisawa, Mituhiro Fukuda, Kazuhiro Kobayashi,
      Kazuhide Nakata, and Maho Nakata.
    \newblock {\em Latest Developments in the SDPA Family for Solving Large-Scale
      SDPs}, pages 687--713.
    \newblock Springer US, Boston, MA, 2012.
    
    \bibitem{Zhang2021}
    Richard~Y. Zhang and Javad Lavaei.
    \newblock Sparse semidefinite programs with guaranteed near-linear time
      complexity via dualized clique tree conversion.
    \newblock {\em Mathematical Programming}, 188(1):351--393, 2021.
    
    \bibitem{Zheng2022}
    Yang Zheng, Aivar Sootla, and Antonis Papachristodoulou.
    \newblock Block factor-width-two matrices and their applications to
      semidefinite and sum-of-squares optimization.
    \newblock {\em {IEEE} Transactions on Automatic Control}, pages 1--1, 2022.
    
    \bibitem{Zohrizadeh2020}
    Fariba Zohrizadeh, Cedric Josz, Ming Jin, Ramtin Madani, Javad Lavaei, and
      Somayeh Sojoudi.
    \newblock A survey on conic relaxations of optimal power flow problem.
    \newblock {\em European Journal of Operational Research}, 287(2):391--409,
      2020.
    
\end{thebibliography}
\end{document}